\documentclass[12pt]{amsart}
\usepackage{amsfonts,amssymb,amscd,amsmath,enumerate,verbatim,calc}
\usepackage{graphicx}

\newtheorem{Theorem}{Theorem}[section]
\newtheorem{Lemma}[Theorem]{Lemma}

\newtheorem{Proposition}[Theorem]{Proposition}
\newtheorem{Remark}[Theorem]{Remark}

\newtheorem{Property}[Theorem]{Property}
\begin{document}
\title{Minimum Fault-Tolerant, local and strong metric dimension of graphs}
\author{Muhammad Salman, Imran Javaid$^*$, Muhammad Anwar Chaudhry}
%\subjclass{Primary: , Secondary: }
\keywords{resolving set, fault-tolerant resolving set, local resolving set, strong resolving set, ILP formulations, convex polytopes.
\indent $^*$ Corresponding author: ijavaidbzu@gmail.com}
%\indent This research of the authors was partially supported by the
%Higher Education Commission of\\
%\indent Pakistan}
\address{Centre for Advanced Studies in Pure and Applied Mathematics,
Bahauddin Zakariya University Multan, Pakistan.}
\address{E-Mail: \{solo33, ijavaidbzu\}@gmail.com, chaudhry@bzu.edu.pk}
%\address{2 College of Computer and Information Systems, Jazan
%University, Jazan, KSA.}
%\address{E-Mail: ahmadsms@gmail.com}

\date{}
\maketitle
%%%---------------------------------------------------------------------
\begin{abstract}
In this paper, we consider three similar optimization problems: the
fault-tolerant metric dimension problem, the local metric dimension
problem and the strong metric dimension problem. These problems have
applications in many diverse areas, including network discovery and
verification, robot navigation and chemistry, etc. We give integer
linear programming formulations of the fault-tolerant metric
dimension problem and the local metric dimension problem. Also, we
study local metric dimension and strong metric dimension of two
convex polytopes $S_n$ and $U_n$.
\end{abstract}

%%---------------------------------------------------------------------------
\section{Introduction.}
The metric dimension problem was introduced independently by Slater
\cite{slater} and Harary and Melter \cite{harary}. Roughly speaking,
the metric dimension of an undirected and connected graph $G$ is the
minimum cardinality of a subset $W$ of vertex set of $G$ with the
property that all the vertices of $G$ are uniquely determined by
their shortest distances to the vertices in $W$. The metric
dimension problem has been widely investigated. Since the complete
survey of all the applications and results is out of scope of this
paper, only some applications and recent results are revealed.

The metric dimension arises in many diverse areas, including
telecommunications \cite{beer}, connected joints in graphs and
chemistry \cite{chartrand}, the robot navigation \cite{khuller} and
geographical routing protocols \cite{liu}, etc. In the area of
telecommunication, especially interesting in the metric dimension
problem application to network discovery and verification
\cite{beer}. Due to its fast dynamic, distributed growth process, it
is hard to obtain an accurate map of the global network. A common
way to obtain such maps is to make certain local measurements at a
small subset of the nodes, and then to combine them in order to
discover the actual graph. Each of these measurements is potentially
quite costly. It is thus a natural objective to minimize the number
of measurements, which still discover the whole graph. That is, to
determine the metric dimension of the graph. In \cite{beer}, simple
greedy strategies were used in a simulation with various types of
randomly generated graphs. The results of the simulation were
presented as two dimensional diagrams displaying the average number
of measurements (cardinality of a resolving set) as a function of
the degree of a particular graph class.

An application of the metric dimension problem in chemistry is
described in \cite{chartrand}. The structure of a chemical compound
can be represented as a labeled graph where the vertex and edge
labels specify the atoms and bond types, respectively. Under the
traditional view, it can be determine whether any two compounds in
the collection share the same functional property at a particular
position. These positions simply reflect uniquely defined atoms
(vertices) of the substructure (common subgraph). It is important to
find smallest number of these positions which is functionally
equivalent to the metric dimension of the given graph. This
observation can be used in drug discovery when it is to be
determined whether the features of a compound are responsible for
its pharmacological activity. For more details see \cite{chartrand}.

An other interesting application of the metric dimension problem
arises in robot navigation \cite{khuller}. Suppose that a robot is
navigating in a space modeled by a graph and wants to know its
current position. It can send a signal to find out how far it is
form each among a set of fixed landmarks. The problem of computing
the minimum number of landmarks and their positions such that the
robot can always uniquely determine its location is equivalent to
the metric dimension problem.

Now, we formally state the metric dimension problem as follows:
Given a simple connected graph $G$ with vertex set $V(G)$ and edge
set $E(G)$. Let $d(u,v)$ denotes the {\it distance} between vertices
$u$ and $v$, $i.e$, the length of a shortest $u-v$ path. A vertex
$w$ of $G$ {\it resolves} the vertices $u$ and $v$ in $G$ if
$d(u,w)\neq d(v,w)$. A subset $W = \{w_1,w_2,\ldots, w_k\}$ of
$V(G)$ is a {\it resolving set} of $G$ if every two distinct
vertices of $G$ are resolved by some vertex of $G$. A {\it metric
basis} of $G$ is a resolving set of the minimum cardinality. The
{\it metric dimension} of $G$, denoted by $\beta(G)$, is the
cardinality of its metric basis.

Metric dimension of several interesting classes of graphs have been
investigated: Grassmann graphs \cite{bailey1}, Johnson and Kneser
graph \cite{bailey2}, cartesian product of graphs \cite{caceres2},
Cayley digraphs \cite{fehr}, convex plytopes \cite{imran},
generalized Petersen graphs \cite{javaid1,javaid2}, Cayley graphs
\cite{javaid3}, silicate networks \cite{manuel}, circulant graphs
\cite{salman}. It also has been shown that some infinite graphs have
infinite metric dimension \cite{caceres1}.

Elements of metric bases were referred to as censors in an
application given in \cite{chartrand6}. If one of the censors does
not work properly, we will not have enough information to deal with
the intruder (fire, thief, etc). In order to overcome this kind of
problems, concept of fault-tolerant metric dimension was introduced
by Hernando {\em et al.} \cite{hernando1}. Fault-tolerant resolving
set provide correct information even when one of the censors is not
working. Roughly speaking, a resolving set is said to be
fault-tolerant if the removal of any element from it keeps it
resolving. Formally, a resolving set $W$ of a graph $G$ is said to
be {\it fault-tolerant} if $W\setminus\{w\}$ is also a resolving set
of $G$, for each $w$ in $W$. The {\it fault-tolerant metric
dimension} (FTMD) of $G$ is the minimum cardinality of a
fault-tolerant resolving set, denoted by $\beta'(G)$. A
fault-tolerant resolving set of cardinality $\beta'(G)$ is called a
{\it fault-tolerant metric basis} (FTMB) of $G$.

A more common problem in graph theory concerns distinguishing every
two neighbors in a graph $G$ by means of some coloring rather than
distinguishing all the vertices of $G$ by graph coloring. Since
distinguishing all the vertices of a connected graph $G$ has been
studied with the aid of distances in $G$. This suggests the topic of
using distances to distinguish the two vertices in each pair of
neighbors only, and thus Okamoto {\em et al.} \cite{okamoto}
introduced the local metric dimension problem, defined as follows: A
subset $W$ of vertex set of a connected graph $G$ is called a {\it
local resolving set} of $G$ if every two adjacent vertices of $G$
are resolved by some element of $W$. A {\it local metric basis} of
$G$ is a local resolving set of the minimum cardinality. The {\it
local metric dimension} of $G$, denoted by $lmd(G)$, is the
cardinality of its local metric basis. Note that each resolving set
of $G$ is {\it vertex-distinguishing} (since it resolves every two
vertices of $G$), and each local resolving set is {\it
neighbor-distinguishing} (since it resolves every two adjacent
vertices of $G$). Thus every resolving set is also a local resolving
set of $G$, so if $G$ is a non-trivial connected graph of order $n$,
then
%\begin{Equation}\label{eq1}
$$\hphantom{aaaaaaaaaaaaaaaaaaaaa} 1\leq lmd(G)\leq \beta(G)\leq n-1.\hphantom{aaaaaaaaaaaaaaaaaaaaa} (1)$$
%\end{Equation}
The strong metric dimension problem was introduced by Seb\"{o} and
Tannier \cite{sebo} and further investigated by Oellermann and
Peters-Fransen \cite{oellermann}. Recently, the strong metric
dimension of distance hereditary graphs has been studied by May and
Oellermann \cite{may}. This concept is defined as follows: A vertex
$w$ {\it strongly resolves} two distinct vertices $u$ and $v$ of $G$
if $u$ belongs to a shortest $v-w$ path or $v$ belongs to a shortest
$u-w$ path, $i.e.$, $d(v,w) = d(v,u)+d(u,w)$ or
$d(u,w)=d(u,v)+d(v,w)$. A subset $S$ of $V(G)$ is a {\it strong
resolving set} of $G$ if every two distinct vertices of $G$ are
strongly resolved by some vertex of $S$. A {\it strong metric basis}
of $G$ is a strong resolving set of the minimum cardinality. The
{\it strong metric dimension} of $G$, denoted by $sdim(G)$, is the
cardinality of its strong metric basis. It is easy to see that if a
vertex $w$ strongly resolves vertices $u$ and $v$, then $w$ also
resolves these vertices. Hence every strong resolving set is a
resolving set and $\beta(G)\leq sdim(G)$.

To determine whether a given set $W\subseteq V(G)$ is a local
(strong) resolving set of $G$, $W$ needs only to be verified for the
vertices in $V(G)\setminus W$ since every vertex $w \in W$ is the
only vertex of $G$ whose distance from $w$ is $0$.

The metric dimension of convex polytopes $S_n, T_n$ and $U_n$, which
are combinations of two graphs of convex polytopes, has been studied
in \cite{imran}. Also the strong metric dimension of $T_n$ has been
studied in \cite{kratica}. In this paper, we study the minimal local
resolving sets and strong resolving sets of the convex polytopes
$S_n$ and $U_n$. We prove that for all $n\geq 3$, $lmd(U_n) = 2$ and
for all $n\geq 3$,
$$ lmd(S_n) = \left\{
              \begin{array}{ll}
           2,         &\,\,\,\,\,\,\ \mbox{if}\ n\ \mbox{is odd},\\
           3,         &\,\,\,\,\,\,\ \mbox{if}\ n\ \mbox{is even},
            \end{array}
             \right.
$$
while the strong metric dimension of both families of convex
polytopes $S_n$ and $U_n$ depends on $n$.

The paper is organized as follows: In section 2, we give integer
linear programming formulations of the fault-tolerant metric
dimension problem and the local metric dimension problem. In section
3 and 4, we get explicit expressions for $lmd(S_n), sdim(S_n),
lmd(U_n)$ and $sdim(U_n)$. In what follows, the indices after $n$
will be taken modulo $n$.

%%%-------------------------------------------------------------------------
\section{Mathematical Programming Formulations}

As described in \cite{cvetk}, it is useful to represent problems of
extremal graph theory as integer linear programming (ILP) problems
in order to use the different well-known optimization techniques.
Following that idea, two integer linear programming formulations of
the metric dimension problem were proposed by Chartrand {\em et al.}
in 2000 \cite{chartrand}, and Currie and Oellermann in 2001
\cite{currie}. Recently, in 2012, Mladenovi\'{c} {\em et al.}
\cite{mlad} proposed a new mathematical programming formulation of
the metric dimension problem with new objective function which
(instead of minimizing the cardinality of a resolving set) minimized
the number of pairs of vertices from $G$ that are not resolved by
vertices of a set with a given cardinality. So the difficulty that
arises when solving the plateaux problem, $i. e.$, problems with a
large number of solutions with the same objective function values,
vanishes with the new objective function. The integer linear
programming formulation of the strong metric dimension problem was
proposed by Kratica {\em et al.} in 2012 \cite{kratica}. To our
knowledge, the following ILP formulations of the FTMD problem and
the local metric dimension problem are new.

%%%------------------------------------------------------------------
\subsection{Fault-Tolerant Metric Dimension Problem}
The following result was proved by Javaid {\em et al.} in
\cite{javaid4}.

\begin{Lemma}{\em \cite{javaid4}}\label{lem} A resolving set $W$ of a graph $G$ is fault-tolerant
if and only if every pair of vertices in $G$ is resolved by at least
two elements of $W$.
\end{Lemma}

Thus, we have the following remark:

\begin{Remark}\label{rem}
Any pair $(u,v)$ of distinct vertices of a connected graph $G$ is
said to be fault-tolerantly resolved in $G$ if for two distinct
vertices $x,y$ of $G$, we have $d(u,x)\neq d(v,x)$ and $d(u,y)\neq
d(v,y)$.
\end{Remark}

Given a simple connected undirected graph $G = (V(G), E(G))$, where
$V(G)=\{1,2,\ldots, n\}$ and $|E(G)|=m$. It is easy to determine the
length $d(u,v)$ of a shortest $u-v$ path for all $u,v\in V(G)$ using
any shortest path algorithm. The coefficient matrix $A$ is defined
as follows:

$$ A_{(u,v),(i,j)} = \left\{
              \begin{array}{ll}
           1,         &\,\,\,\,\,\,\ d(u,i)\neq d(v,i)\ \mbox{and}\ d(u,j)\neq d(v,j),\\
           0,         &\,\,\,\,\,\,\
           d(u,i)=d(v,i)\ \mbox{and}\ d(u,j)=d(v,j),\hphantom{aaaaaaaaaaaaaaaaaaa}(2)
            \end{array}
             \right.
$$
where $1\leq u< v\leq n$, $1\leq i< j\leq n$. Variable $x_i$
described by $(3)$ determines whether vertex $i$ belongs to a
fault-tolerant resolving set $W$ or not. Similarly, $y_{ij}$
determines whether both $i,j$ are in $W$.
$$ x_i = \left\{
              \begin{array}{ll}
           1,         &\,\,\,\,\,\,\ i\in W,\\
           0,         &\,\,\,\,\,\,\
           i\not \in W.\ \hphantom{aaaaaaaaaaaaaaaaaaaaaaaaaaaaaaaaaaaaaaaaaaaaaaaaa}(3)
            \end{array}
             \right.
$$
$$ y_{ij} = \left\{
              \begin{array}{ll}
           1,         &\,\,\,\,\,\,\ i,j\in W,\\
           0,         &\,\,\,\,\,\,\ \mbox{otherwise}.\hphantom{aaaaaaaaaaaaaaaaaaaaaaaaaaaaaaaaaaaaaaaaaaaaaa}(4)
            \end{array}
             \right.
$$

The ILP model of the FTMD problem can now be formulated as:
$$\mbox{Minimze}\ f(x_1,x_2,\ldots, x_n) = \sum \limits_{k=1}^n x_k \hphantom{aaaaaaaaaaaaaaaaaaaaaaaaaaaaaaaaaaaaaa}(5)$$
subject to:
$$\sum \limits_{i=1}^{n-1} \sum \limits_{j=i+1}^{n} A_{(u,v),(i,j)}\ y_{ij} \geq 1,\qquad 1\leq u< v\leq n, \hphantom{aaaaaaaaaaaaaaaaaaaaaaaaaaa}(6)$$
$$y_{ij}\leq \frac{1}{2}x_i+\frac{1}{2}x_j, \qquad 1\leq i< j\leq n, \hphantom{aaaaaaaaaaaaaaaaaaaaaaaaaaaaaaaaaaa}(7)$$
$$y_{ij}\geq x_i+x_j -1, \qquad 1\leq i< j\leq n,\ \hphantom{aaaaaaaaaaaaaaaaaaaaaaaaaaaaaaaaaa}(8)$$
$$y_{ij}\in \{0,1\},\qquad x_k\in \{0,1\}, \qquad 1\leq i< j\leq n,\qquad 1\leq k\leq n. \hphantom{aaaaaaaaaaaaa}(9)$$
Note that, the ILP model $(5)$-$(9)$ has $n+{n\choose 2}$ variables
and $3{n\choose 2}$ linear constraints. The following proposition
shows that each feasible solution of $(6)$-$(9)$ defines a
fault-tolerant resolving set of $G$ and vice-versa.

\begin{Proposition}\label{prop}
$W$ is a fault-tolerant resolving set of $G$ if and only if
constraints $(6)$-$(9)$ are satisfied.
\end{Proposition}

\begin{proof}
$(\Rightarrow)$ Suppose that $W$ is a fault-tolerant resolving set
of $G$. Then for each $u, v \in V(G),\ u\neq v$, there exist $i, j
\in W$ ($i.e.,\ y_{ij} = 1$), $i\neq j$, such that $d(u, i) \neq
d(v, i)$ and $d(u, j)\neq d(v, j)$. Without loss of generality, we
may assume that $u< v$ and $i < j$. It follows that $A_{(u,v),(i,j)}
= 1$, and consequently constraints $(6)$ are satisfied. Constraints
$(7)$-$(9)$ are obviously satisfied since $i, j \in W$ implies that
$x_i = x_j = y_{ij} = 1$.

$(\Leftarrow)$ According to $(3)$, $W = \{i \in \{1, 2,\ldots, n\}\
|\ x_i = 1\}$. For all $i, j \in W$, from $(8)$ and $(9)$ it follows
that $y_{ij} =1$ because $y_{ij}\geq x_i +x_j -1=1$ and by $(9)$,
$y_{ij}$ is a binary variable. If $i$ or $j$ is not in $W$, then
constraints $(7)$ imply that $y_{ij}\leq \frac{1}{2} x_i+\frac{1}{2}
x_j\leq \frac{1}{2}$. Since $y_{ij}$ is a binary variable, it
follows that $y_{ij}=0$. Therefore, $y_{ij} = 1$ if and only if $i,
j \in W$. If constraints $(6)$ are satisfied, then for each $1\leq
u<v\leq n$, there exist $i, j \in \{1,\ldots, n\},\ i < j$, such
that $A_{(u,v),(i,j)} y_{ij}\geq 1$, which implies that $y_{ij} = 1\
(i.e., i, j \in W)$ and $A_{(u,v),(i,j)} = 1\ (i.e., d(u, i)\neq
d(v, i)\ \mbox{and}\ d(u, j)\neq d(v, j))$. It follows that the set
$W$ is a fault-tolerant resolving set of $G$.
\end{proof}

\begin{Remark}\label{remark}
The ILP formulation of the FTMD problem generally looks like the ILP
formulation of the minimum doubly resolving set (MDRS) problem
proposed by Kratica {et al.} \cite{kratica1}. But the difference is
in finding the entries $A_{(u,v),(i,j)}$ of the coefficient matrix
$A$. For instance, if $G = P_4:\ 1,2,3,4$ (path on four vertices).
Then in the case of MDRS problem, $A_{(1,2),(3,4)} = 0$ (by the
definition of the coefficient matrix given in \cite{kratica1}),
where as in the case of FTRS problem, $A_{(1,2),(3,4)} = 1$.
\end{Remark}
%%%------------------------------------------------------------------
\subsection{Local Metric Dimension Problem}
Let $u$ be a vertex of a graph $G$. The {\it open neighborhood} of
$u$ is $N(u)= \{v \in V(G)\ |\ v\ \mbox{is adjacent with}\ u\
\mbox{in}\ G\}$ and the {\it closed neighborhood} of $u$ is $N[u]=
N(u) \cup \{u\}$. Now, from the definition of local resolving set,
we have the following proposition:

\begin{Proposition}\label{prop1}
A subset $W$ of $V(G)$ of a non-trivial connected graph $G$ is local
resolving set if and only if for all $u\in V(G)$ and for each $v\in
N(u)$, $d(u,w)\neq d(v,w)$ for some $w\in W$.
\end{Proposition}

\begin{proof}
$(\Rightarrow)$ Suppose that $W$ is a local resolving set of $G$.
Then by definition, for any two adjacent vertices $u$ and $v$ in
$G$, $i.e.$, for all $u\in V(G)$ and for each $v\in N(u)$, there
exists a vertex $w$ in $W$ such that $d(u,w)\neq d(v,w)$.

$(\Leftarrow)$ If for all $u\in V(G)$ and for each $v\in N(u)$,
$d(u,w)\neq d(v,w)$ for some $w\in W$, then every two adjacent
vertices are resolved by some vertex $w$ of $W$, which implies that
$W$ is a local resolving set of $G$.
\end{proof}

Given a simple connected undirected graph $G = (V(G), E(G))$, where
$V(G)=\{1,2,\ldots, n\}$ and $|E(G)|=m$. It is easy to determine the
length $d(u,v)$ of a shortest $u-v$ path for all $u,v\in V(G)$ using
any shortest path algorithm. The coefficient matrix $A$ is defined
as follows:

$$ A_{(u,v),i} = \left\{
              \begin{array}{ll}
           1,         &\,\,\,\,\,\,\ d(u,i)\neq d(v,i),\\
           0,         &\,\,\,\,\,\,\
           d(u,i)=d(v,i),\hphantom{aaaaaaaaaaaaaaaaaaaaaaaaaaaaaaaaaaaa}(10)
            \end{array}
             \right.
$$
where $1\leq u\leq n$, $1\leq i\leq n$ and $v\in N(u)$ with $v>u$.

Variable $x_i$ described by $(11)$ determines whether vertex $i$
belongs to a local resolving set $W$ or not.
$$ x_i = \left\{
              \begin{array}{ll}
           1,         &\,\,\,\,\,\,\ i\in W,\\
           0,         &\,\,\,\,\,\,\
           i\not \in W.\hphantom{aaaaaaaaaaaaaaaaaaaaaaaaaaaaaaaaaaaaaaaaaaaaaaaa}(11)
            \end{array}
             \right.
$$

The ILP model of the local metric dimension problem can now be
formulated as:
$$\mbox{Minimze}\ f(x_1,x_2,\ldots, x_n) = \sum \limits_{i=1}^n x_i \hphantom{aaaaaaaaaaaaaaaaaaaaaaaaaaaaaaaaaaaa}(12)$$
subject to:
$$\sum \limits_{i=1}^n A_{(u,v),i}\ x_i \geq 1,\qquad 1\leq u\leq n\,\ \mbox{and}\,\ v\in N(u)\,\, \mbox{with}\,\ v> u, \hphantom{aaaaaaaaaaaaaa}(13)$$
$$x_i\in \{0,1\},\qquad 1\leq i\leq n. \hphantom{aaaaaaaaaaaaaaaaaaaaaaaaaaaaaaaaaaaaaaaaaaa}(14)$$
Note that, the ILP model $(12)$-$(14)$ has $n$ variables and $m$
linear constraints. The following proposition shows that each
feasible solution of $(13)$ and $(14)$ defines a local resolving set
of $G$ and vice-versa.

\begin{Proposition}\label{prop2}
$W$ is a local resolving set of $G$ if and only if constraints
$(13)$ and $(14)$ are satisfied.
\end{Proposition}

\begin{proof}
$(\Rightarrow)$ Suppose that $W$ is a local resolving set of $G$.
Then by Proposition \ref{prop1}, for all $u\in V(G)$ and for each
$v\in N(u)$, there exists a vertex $i$ in $W$ $(i.e., x_i = 1)$ such
that $d(u,i)\neq d(v,i)$. It follows that $A_{(u,v),i} = 1$, and
consequently constraints $(13)$ are satisfied. Constraints $(14)$
are obviously satisfies by the definition of variable $x_i$.

$(\Leftarrow)$ According to $(11)$, $W = \{i\in \{1,2,\ldots, n\}\
|\ x_i = 1\}$. If constraints $(13)$ are satisfied, then for all
$u\in V(G)$ and for each $v\in N(u)$ with $v>u$, there exists $i\in
\{1,2,\ldots, n\}$ such that $A_{(u,v),i}\ x_i \geq 1$. This implies
that $x_i = 1\ (i.e., i\in W)$ and $A_{(u,v),i} = 1\ (i.e.,
d(u,i)\neq d(v,i))$. It follows, by Proposition \ref{prop1}, that
$W$ is a local resolving set of $G$.
\end{proof}

%%%--------------------------------------------------------------------------
\section{Convex Polytopes $S_{n}$}
The convex polytopes $S_n$, $n\geq 3$, \cite{imran} (see Figure
\ref{fig1}) consists of $2n$ 3-sided faces, $2n$ 4-sided faces and a
pair of $n$-sided faces obtained by the combination of a convex
polytope $R_n$ and a prism $D_n$ having vertex and edge sets as:
$$V(S_n)=\{a_i,b_i,c_i,d_i\ |\ 1\leq i\leq n\},\hphantom{aaaaaaaaaaaaaaaaaaaaaaaaaaaaaaaaaaaaaaaaaaaaaaa}$$
$$E(S_n)=\{a_ia_{i+1}, b_ib_{i+1}, c_ic_{i+1}, d_id_{i+1},a_{i+1}b_i, a_ib_i, b_ic_i, c_id_i\ |\ 1\leq i\leq n\}.\hphantom{aaaaaaaaaaaaaaaaaaa}$$

\begin{figure}[h]
        \centerline
        {\includegraphics[width=13cm]{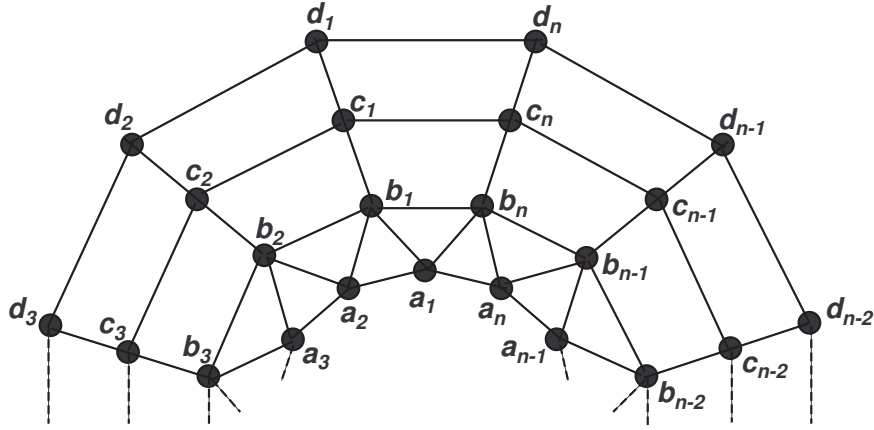}}
        \caption{The graph of convex polytope
        $S_n$}\label{fig1}
\end{figure}
The metric dimension of $S_n$ was studied in \cite{imran}. In this
section, we show that $lmd(S_n) = 2$ when $n$ is odd and
$lmd(S_n)=3$ when $n$ is even. Moreover, we show that $sdim(S_n) =
n$ when $n$ is odd and $sdim(S_n) = \frac{3n}{2}$ when $n$ is even.
The main results of this section are the following:

\begin{Theorem}\label{th1}
For any convex polytope $S_n$, $n\geq 3$, we have $$ lmd(S_n) =
\left\{
              \begin{array}{ll}
           2,         &\,\,\,\,\,\,\ \mbox{if}\ n\ \mbox{is odd},\\
           3,         &\,\,\,\,\,\,\ \mbox{if}\ n\ \mbox{is even}.
            \end{array}
             \right.
$$
\end{Theorem}

\begin{Theorem}\label{th2}
For any convex polytope $S_n$, $n\geq 3$, we have
$$ sdim(S_n) =
\left\{
              \begin{array}{ll}
           n,         &\,\,\,\,\,\,\ \mbox{if}\ n\ \mbox{is odd},\\
           \frac{3n}{2},         &\,\,\,\,\,\,\ \mbox{if}\ n\ \mbox{is
           even}.
            \end{array}
             \right.
$$
\end{Theorem}

For each fixed $i\in \{a,b,c,d\}$, let $C_i$ denotes the cycle
induced by the vertices $i_1,i_2,\ldots, i_n$ in $S_n$. Next we
prove the several lemmas which support the proofs of Theorem
\ref{th1} and Theorem \ref{th2}.

\begin{Lemma}\label{lem1}
For $n=2k$, $k\geq 2$, if $lmd(S_n)=2$, then any local metric basis
of $S_n$ does not contain both the vertices of the same cycle $C_i$,
$i\in \{a,b,c,d\}$.
\end{Lemma}

\begin{proof}
Suppose contrarily that $W = \{u,v\}$ be a local matric basis of
$S_n$ with $u,v\in V(C_i)$, $i\in \{a,b,c,d\}$. Then for fixed $i\in
\{a,b,c,d\}$, there is no loss of generality in assuming that $u =
i_1$ and $v=i_j,\ 2\leq j\leq n$. This gives that for any two
adjacent vertices $x$ and $y$ of $S_n$ such that
$$ x,y \in
\left\{
              \begin{array}{ll}
           N[a_1],         &\,\,\,\,\,\,\ \mbox{when}\,\ 2\leq j\leq k+1,\\
           N[a_2],         &\,\,\,\,\,\,\ \mbox{when}\,\ k+2\leq j\leq
           n, \hphantom{aaaaaaaaaaaaaaaaaaaaaaaaaaaaaaaaaaaaaa}
            \end{array}
             \right.
$$
we have $d(x,w) = d(y,w)$ for all $w\in W$, a contradiction to the
fact that $lmd(S_n) = 2$.
\end{proof}

\begin{Lemma}\label{lem2}
For $n=2k$, $k\geq 2$, if $lmd(S_n)=2$, then any local metric basis
$W$ of $S_n$ does not has the property that $W$ contains one vertex
from $C_i$ and the other one from $C_j$ ($j\neq i$), where $i,j\in
\{a,b,c,d\}$.
\end{Lemma}

\begin{proof}
Suppose contrarily that $W = \{u,v\}$ be a local matric basis of
$S_n$ with $u\in V(C_i)$ and $v\in V(C_j)$ $(j\neq i)$, $i,j\in
\{a,b,c,d\}$. Then we have the following three cases:\\

\noindent Case 1: When $u\in C_a$ and $v\in C_i,\ i\in \{b,c,d\}$.
Without loss of generality, we \hphantom{aaaaaaaa}assume that $u =
a_1$ and $v = i_j$, $1\leq j\leq n$. Then for $1\leq j\leq k$, there
\hphantom{aaaaaaa} exists a vertex $b_{j+k}$ in $N[a_{j+k}]$ such
that $d(a_{j+k}, w) = d(b_{j+k},w)$ for all \hphantom{aaaaaaa} $w\in
W$, a contradiction. Also for $k+1\leq j\leq n$, there exists a
vertex \hphantom{aaaaaaa} $a_{j-k+1}$ in $N[b_{j-k}]$ such that
$d(a_{j-k+1}, w) = d(b_{j-k},w)$ for all $w\in W$, a
\hphantom{aaaaaaa} contradiction.

\noindent Case 2: When $u\in C_b$ and $v\in C_i,\ i\in \{c,d\}$.
Without loss of generality, we \hphantom{aaaaaaaa}assume that $u =
b_1$ and $v = i_j$, $1\leq j\leq n$. Then there exist two vertices
\hphantom{aaaaaaaa}$x$ and $y$ in $N[u]$ such that $d(x, w) =
d(y,w)$ for all $w\in W$, a contradiction.

\noindent Case 3: When $u\in C_c$ and $v\in C_d$. Without loss of
generality, we assume that \hphantom{aaaaaaa} $u = c_1$ and $v =
d_j$, $1\leq j\leq n$. Then there exist two vertices $x$ and $y$ in
\hphantom{aaaaaaa} $N[u]$ such that $d(x, w) = d(y,w)$ for all $w\in
W$, a contradiction.
\end{proof}

\begin{Lemma}\label{lem3}
For $n = 2k+1$, $k\geq 2$, $W = \{a_1,a_{k+1}\}$ is a local
resolving set of $S_n$.
\end{Lemma}

\begin{proof}
We show that for all $u= i_j \in V(S_n)$, $i\in \{a,b,c,d\}$, $1\leq
j\leq n$, and for each $v\in N(u)$, $d(u,a_1)-d(v,a_1)\neq 0$ or
$d(u,a_{k+1})-d(v, a_{k+1}) \neq 0$. Then Proposition \ref{prop1}
will concludes that the set $\{a_1,a_{k+1}\}$ is a local resolving
set of $S_n$. First note that, if $u = a_1$ or $u = a_{k+1}$, then
for each $v\in N(u)$, $d(u,u)=0\neq 1=d(v,u)$. Further, note that
$$ d(a_j,a_1) =
\left\{
              \begin{array}{ll}
           j-1,         &\,\,\,\,\,\,\ 1\leq j\leq k+1,\\
           2k-j+2,       &\,\,\,\,\,\,\ k+2\leq j\leq
           2k+1, \hphantom{aaaaaaaaaaaaaaaaaaaaaaaaaaaaaaaaaaaaaa}
            \end{array}
             \right.
$$
$$ d(b_j,a_{1}) =
\left\{
              \begin{array}{ll}
           j,         &\,\,\,\,\,\,\ 1\leq j\leq k,\\
           2k-j+2,       &\,\,\,\,\,\,\ k+1\leq j\leq
           2k+1, \hphantom{aaaaaaaaaaaaaaaaaaaaaaaaaaaaaaaaaaaaaa}
            \end{array}
             \right.
$$
$$ d(c_j,a_1) =
\left\{
              \begin{array}{ll}
           j+1,         &\,\,\,\,\,\,\ 1\leq j\leq k,\\
           2k-j+3,       &\,\,\,\,\,\,\ k+1\leq j\leq
           2k+1, \hphantom{aaaaaaaaaaaaaaaaaaaaaaaaaaaaaaaaaaaaaa}
            \end{array}
             \right.
$$
$$ d(d_j,a_1) =
\left\{
              \begin{array}{ll}
           j+2,         &\,\,\,\,\,\,\ 1\leq j\leq k,\\
           2k-j+4,       &\,\,\,\,\,\,\ k+1\leq j\leq
           2k+1, \hphantom{aaaaaaaaaaaaaaaaaaaaaaaaaaaaaaaaaaaaaa}
            \end{array}
             \right.
$$
$$ d(a_j,a_{k+1}) =
\left\{
              \begin{array}{ll}
           k-j+1,         &\,\,\,\,\,\,\ 1\leq j\leq k+1,\\
           j-k-1,       &\,\,\,\,\,\,\ k+2\leq j\leq
           2k+1, \hphantom{aaaaaaaaaaaaaaaaaaaaaaaaaaaaaaaaaaaaaa}
            \end{array}
             \right.
$$
$$ d(b_j,a_{k+1}) =
\left\{
              \begin{array}{ll}
           k-j+1,         &\,\,\,\,\,\,\ 1\leq j\leq k,\\
           j-k,       &\,\,\,\,\,\,\ k+1\leq j\leq
           2k+1, \hphantom{aaaaaaaaaaaaaaaaaaaaaaaaaaaaaaaaaaaaaa}
            \end{array}
             \right.
$$
$$ d(c_j,a_{k+1}) =
\left\{
              \begin{array}{ll}
           k-j+2,         &\,\,\,\,\,\,\ 1\leq j\leq k,\\
           j-k+1,       &\,\,\,\,\,\,\ k+1\leq j\leq
           2k+1, \hphantom{aaaaaaaaaaaaaaaaaaaaaaaaaaaaaaaaaaaaaa}
            \end{array}
             \right.
$$
$$ d(d_j,a_{k+1}) =
\left\{
              \begin{array}{ll}
           k-j+3,         &\,\,\,\,\,\,\ 1\leq j\leq k,\\
           j-k+2,       &\,\,\,\,\,\,\ k+1\leq j\leq
           2k+1. \hphantom{aaaaaaaaaaaaaaaaaaaaaaaaaaaaaaaaaaaaaa}
            \end{array}
             \right.
$$
Now, according to the above listed distances, the following four
cases conclude the proof.\\

\noindent Case 1: For each neighborhood $N(a_j) =\{a_{j-1}, a_{j+1},
b_i, b_{j-1}\}$ of $a_j$, $2\leq j\leq 2k+ \hphantom{aaaaaaa} 1\
(j\neq k+1)$, we have

$$\hphantom{aaaaaaa} |d(a_j,a_{1})-d(v,a_1)| =
\left\{
              \begin{array}{ll}
           0,       &\,\,\,\,\,\,\ 2\leq j\leq k;\ v= b_{j-1},\\
           0,       &\,\,\,\,\,\,\ k+2\leq j\leq 2k+1;\ v= b_{j},\\
           1,         &\,\,\,\,\,\,\ 2\leq j\leq k;\ v\in \{a_{j-1},a_{j+1},b_j\},\\
           1,         &\,\,\,\,\,\,\ k+2\leq j\leq 2k+1;\ v\in \{a_{j-1},a_{j+1},b_{j-1}\},
            \end{array}
             \right.
$$
$$\hphantom{aaaaaaa} |d(a_j,a_{k+1})-d(v,a_{k+1})| =
\left\{
              \begin{array}{ll}
           0,       &\,\,\,\,\ 2\leq j\leq k;\ v= b_{j},\\
           0,       &\,\,\,\,\ k+2\leq j\leq 2k+1;\ v= b_{j-1},\\
           1,         &\,\,\,\,\ 2\leq j\leq k;\ v\in \{a_{j-1},a_{j+1},b_{j-1}\},\\
           1,         &\,\,\,\,\ k+2\leq j\leq 2k+1;\ v\in
\{a_{j-1},a_{j+1},b_{j}\}.
            \end{array}
             \right.
$$
\noindent Case 2: For each neighborhood $N(b_j) =\{a_{j}, a_{j+1},
b_{j-1}, b_{j+1}, c_j\}$ of $b_j$, $1\leq j\leq \hphantom{aaaaaaaa}
2k+1\ (j\neq k+1)$, we have

$$\hphantom{aaa} |d(b_j,a_{1})-d(v,a_1)| =
\left\{
              \begin{array}{ll}
           0,         &\,\,\,\,\,\,\ j=1;\ v\in \{a_{j+1},b_{j-1}\},\\
           0,       &\,\,\,\,\,\,\ 2\leq j\leq k;\ v= a_{j+1},\\
           0,         &\,\,\,\,\,\,\ k+2\leq j\leq 2k;\ v = a_j,\\
           0,       &\,\,\,\,\,\,\ j = 2k+1;\ v\in \{a_j,b_{j+1}\},\\
           1,         &\,\,\,\,\,\,\ j=1;\ v\in \{a_{j},b_{j+1},c_j\},\\
           1,         &\,\,\,\,\,\,\ 2\leq j\leq k;\ v\in N(b_j)\setminus\{a_{j+1}\},\\
           1,         &\,\,\,\,\,\,\ j=k+1;\ \mbox{for all}\ v\in N(b_j),\\
           1,         &\,\,\,\,\,\,\ k+2\leq j\leq 2k;\ v\in N(b_j)\setminus\{a_{j}\},\\
           1,       &\,\,\,\,\,\,\ j = 2k+1;\ v\in \{a_{j+1},b_{j-1},c_{j}\},
            \end{array}
             \right.
$$
$$\hphantom{aaaaaaaa} |d(b_j,a_{k+1})-d(v,a_{k+1})| =
\left\{
              \begin{array}{ll}
           0,       &\,\,\,\,\,\,\ 1\leq j\leq k-1;\ v= a_{j},\\
           0,       &\,\,\,\,\,\,\ j = k;\ v\in \{a_j,b_{j+1}\},\\
           0,       &\,\,\,\,\,\,\ j = k+1;\ v\in \{a_{j+1},b_{j-1}\},\\
           0,         &\,\,\,\,\,\,\ k+2\leq j\leq 2k;\ v = a_{j+1},\\
           1,         &\,\,\,\,\,\,\ 1\leq j\leq k-1;\ v\in N(b_j)\setminus \{a_{j}\},\\
           1,       &\,\,\,\,\,\,\ j = k;\ v\in \{a_{j+1},b_{j-1},c_j\},\\
           1,       &\,\,\,\,\,\,\ j = k+1;\ v\in \{a_{j},b_{j+1},c_j\},\\
           1,         &\,\,\,\,\,\,\ k+2\leq j\leq 2k;\ v\in N(b_j)\setminus\{a_{j+1}\},\\
           1,         &\,\,\,\,\,\,\ j=2k+1;\ \mbox{for all}\ v\in N(b_j).\\
            \end{array}
             \right.
$$
\noindent Case 3: For each neighborhood $N(c_j) =\{b_{j}, c_{j-1},
c_{j+1}, d_j\}$ of $c_j$, $1\leq j\leq 2k+ \hphantom{aaaaaaaa} 1\
(j\neq k+1)$, we have $d(c_1,a_1) - d(c_{2k+1},a_1) = 0 =
d(c_k,a_{k+1})- \hphantom{aaaaaaaa} d(c_{k+1},a_{k+1})$ only. But
than $|d(c_1,a_{k+1})-d(c_{2k+1},a_{k+1})| = 1 = |d(c_k,a_1)-
\hphantom{aaaaaaaa} d(c_{k+1},a_1)|$.

\noindent Case 4: For each neighborhood $N(d_j) =\{c_{j}, d_{j-1},
d_{j+1}\}$ of $d_j$, $1\leq j\leq 2k+ \hphantom{aaaaaaaa} 1\ (j\neq
k+1)$, we have $d(d_1,a_1) - d(d_{2k+1},a_1) = 0 = d(d_k,a_{k+1})-
\hphantom{aaaaaaaa} d(d_{k+1},a_{k+1})$ only. But than
$|d(d_1,a_{k+1})-d(d_{2k+1},a_{k+1})| = 1 = |d(d_k,a_1)-
\hphantom{aaaaaaaa} d(d_{k+1},a_1)|$.
\end{proof}

For a vertex $v$ in $G$, the {\it eccentricity}, $ecc(v)$, is the
maximum distance between $v$ and any other vertex of $G$. The {\it
diameter} of $G$, denoted by $diam(G)$, is the maximum eccentricity
of a vertex $v$ in $G$. The following lemma and two properties,
proved by Kratica {\em et al.} \cite{kratica}, will be used in the
sequel.

\begin{Lemma}{\em \cite{kratica}}\label{lem4}
Let $u, v \in V(G),\ u\neq v$, and\\
$(i)\ d(w,v)\leq d(u,v)$ for each $w\in N(u)$ and\\
$(ii)\ d(u,w)\leq d(u,v)$ for each $w\in N(v)$.\\
Then there does not exist vertex $x\in V(G)$, $x\neq u,v$, that
strongly resolves the vertices $u$ and $v$.
\end{Lemma}

\begin{Property}{\em \cite{kratica}}\label{pro1}
If $S$ is a strong resolving resolving set of $G$, then for every
two distinct vertices $u,v\in V(G)$ which satisfy conditions $(i)$
and $(ii)$ of Lemma \ref{lem4}, we have $u\in S$ or $v\in S$.
\end{Property}

\begin{Property}{\em \cite{kratica}}\label{pro2}
If $S$ is a strong resolving resolving set of $G$, then for every
two distinct vertices $u,v\in V(G)$ such that $d(u,v) = diam(G)$, we
have $u\in S$ or $v\in S$.
\end{Property}

\begin{Lemma}\label{lem5}
For $n = 2k+1,\ k\geq 1$, if $S$ is a strong resolving set of $S_n$,
then $|S|\geq n$.
\end{Lemma}

\begin{proof}
Let us consider the pair $(a_i,d_{i+k})$ of vertices of $S_n$ for
$i=1,2,\ldots, n$. Then it is easy to see that $d(a_i,d_{i+k}) =
k+3$. Since $diam(S_n) = k+3$ so according to the Property
\ref{pro2}, $a_i\in S$ or $d_{i+k}\in S$ for all $i=1,2\ldots,n$.
Therefore $|S|\geq n$.
\end{proof}

\begin{Lemma}\label{lem6}
For $n = 2k+1,\ k\geq 1$, the subset $\{d_i\ |\ i = 1,2,\ldots,n\}$
of $V(S_n)$ is a strong resolving set of $S_n$.
\end{Lemma}

\begin{proof}
Let us prove that for each $i=1,2,\ldots,n$, the vertex $d_i$
strongly resolves the pairs $(c_i,a_j), (c_i,b_j), (c_i,c_j)_{(i\neq
j)}, (b_i,a_j), (b_i,b_j)_{(i\neq j)}$ and $(a_i,a_j)_{(j\neq
i,i+1,\ldots,i+k)}$, where $j = 1,2,\ldots,n$. It is easy to see
that $d(d_i,c_i) = 1$, $d(d_i,b_i) = 2 = d(d_i,c_i)+1$ and
$$ d(b_j,c_i) = d(c_j,d_i) =
\left\{
              \begin{array}{ll}
           j-i-1,         &\,\,\,\,\,\,\ i\leq j\leq i+k,\\
           n-j+i+1,       &\,\,\,\,\,\,\ i+k+1\leq j\leq
           i+n-1. \hphantom{aaaaaaaaaaaaaaaaaaaaaaaaaaaaaaaaaaaaaa}
            \end{array}
             \right.
$$
Note that
$$(i)\qquad\,\,\,\ d(c_j,c_i) = d(c_j,d_i)-1,\hphantom{aaaaaaaaaaaaaaaaaaaaaaaaaaaaaaaaaaaaaaaaaaaaaaaa}$$
$$\qquad \Rightarrow\,\ d(d_i,c_j) = d(c_i,c_j)+1 = d(d_i,c_i)+ d(c_i,c_j).\hphantom{aaaaaaaaaaaaaaaaaaaaaaaa} (15)$$
$$(ii)\qquad \,\ d(b_j,d_i) = d(c_j,d_i)+1 = d(b_j,c_i)+ d(c_i,d_i),\hphantom{aaaaaaaaaaaaaaaaaaaaaaaaaaaaaa}$$
$$\qquad \Rightarrow\,\ d(d_i,b_j) = d(d_i,c_i)+d(c_i,b_j).\hphantom{aaaaaaaaaaaaaaaaaaaaaaaaaaaaaaaaaaaa} (16)$$
$$(iii)\qquad d(b_j,b_i) = d(b_j,c_i)-1,\hphantom{aaaaaaaaaaaaaaaaaaaaaaaaaaaaaaaaaaaaaaaaaaaaaaaa}$$
so equation $(16)$ becomes
$$\qquad\,\ d(d_i,b_j) = d(d_i,c_i)+d(b_i,b_j)+1 = d(d_i,b_i)+d(b_i,b_j).\hphantom{aaaaaaaaaaaaaaaaaa} (17)$$
Equations $(15), (16)$ and $(17)$ conclude that the pairs
$(c_i,c_j),(c_i,b_j)$ and $(b_i,b_j)$ are strongly resolved by
$d_i$.

\noindent Now, consider the following shortest paths between $d_i$
and $a_{i+k+1}$:\\
\indent $\bullet\,\ P_1:\
d_i,c_i,b_i,a_i,a_{i-1},a_{i-2},\ldots,a_{i+k+2},a_{i+k+1}$;\\
\indent $\bullet\,\ P_2:\
d_i,c_i,b_i,a_{i+1},a_{i+2},\ldots,a_{i+k},a_{i+k+1}$.\\
Then one can easily see that for $j = i,i+k+1,i+k+2,\ldots,i+n-1$,
the pairs $(c_i,a_j), (b_i,a_j)$ and the pair $(a_i,a_j),\ j\neq
i,i+1,i+2,\ldots,i+k$, are strongly resolved by $d_i$ according to
the shortest $d_i-a_j$ path $P_1$ which contains the vertices
$a_i,b_i, c_i$; and for $j = i+1,i+2,\ldots,i+k$, the pairs
$(c_i,a_j)$ and $(b_i,a_j)$ are strongly resolved by $d_i$ according
to the shortest $d_i-a_j$ path $P_2$ which contains the vertices
$b_i$ and $c_i$.

Finally, the remaining pairs $(a_i,a_j),\ j= i+1,i+2,\ldots,i+k$,
are strongly resolved by $d_{i+k}$ because of the shortest
$d_{i+k}-a_i$ path $d_{i+k},c_{i+k}, b_{i+k},
a_{i+k},a_{i+k-1},\ldots,a_{i+1},a_i$, which contains the vertex
$a_j$.
\end{proof}

\begin{Lemma}\label{lem7}
For $n = 2k,\ k\geq 2$, if $S$ is a strong resolving set of $S_n$,
then $|S|\geq \frac{3n}{2}$.
\end{Lemma}

\begin{proof}
Consider the pair $(b_i,d_{i+k})$ of vertices of $S_n$ for
$i=1,2,\ldots, n$. Then $d(a_i,d_{i+k}) \\= k+2$. As $diam(S_n) =
k+2$, so according to the Property \ref{pro2}, $b_i\in S$ or
$d_{i+k}\in S$ for all $i=1,2\ldots,n$. Moreover, the vertices in
pair $(a_i,a_{i+k}),\ i=1,2,\ldots,k$, satisfy both the conditions
of Lemma \ref{lem4}. So according to the Property \ref{pro1},
$a_i\in S$ or $a_{i+k}\in S$ for all $i=1,2\ldots,k$. Therefore $|S|
\geq n+k =\frac{3n}{2}$.
\end{proof}

\begin{Lemma}\label{lem8}
For $n = 2k,\ k\geq 2$, the subset $\{d_i, a_{i'}\ |\ i =
1,2,\ldots,n;\ i'=1,2,\ldots,k\}$ of $V(S_n)$ is a strong resolving
set of $S_n$.
\end{Lemma}

\begin{proof}
First we prove that for each $i=1,2,\ldots,n$, the vertex $d_i$
strongly resolves the pairs $(c_i,b_j), (c_i,c_j)_{(i\neq
j)},(b_i,b_j)_{(i\neq j)}$ for all $j=1,2\ldots,n$ and the pairs
$(c_i,a_j), (b_i,a_j)$ for $j = k+1,k+2,\ldots,n$. To this end, we
consider the following shortest paths:\\
\indent $\bullet\,\ P_1:\
d_i,c_i,c_{i+1},\ldots,c_{i+k-1},c_{i+k}$;\\
\indent $\bullet\,\ P_2:\
d_i,c_i,c_{i-1},c_{i-2},\ldots,c_{i+k+1},c_{i+k}$;\\
\indent $\bullet\,\ P_3:\ d_i,c_i,b_i,b_{i+1},\ldots,b_{i+k-1},b_{i+k}$;\\
\indent $\bullet\,\ P_4:\
d_i,c_i,b_i,b_{i-1},b_{i-2},\ldots,b_{i+k+1},b_{i+k}$;\\
\indent $\bullet\,\ P_5:\ d_i,c_i,b_i,a_{i+1},\ldots,a_{i+k-1},a_{i+k}$;\\
\indent $\bullet\,\ P_6:\
d_i,c_i,b_i,a_i,a_{i-1},a_{i-2},\ldots,a_{i+k+2},a_{i+k+1}$.\\
Note that, each above mentioned pair is strongly resolved by $d_i$
because of the existence of a shortest $d_i-v_j$ path $(v\in
\{a,b,c\})$ as shown in Table \ref{table1}.

\begin{table}
 \begin{center}
\begin{tabular}{|c|c|c|}
               \hline
               pair & for $i+1\leq j\leq i+k$  & for $j = i,i-1,i-2,\ldots,i+k+1$  \\
               \hline
               $(c_i,c_j)$ & $v = c,\ P_1$ contains $c_i$ & $v = c,\ P_2$ contains $c_i$\\
               $(b_i,b_j),(c_i,b_j)$ & $v=b,\ P_3$ contains $b_i,c_i$ & $v=b,\ P_4$ contains $b_i,c_i$ \\
               $(b_i,a_j), (c_i,a_j)$ & $v=a,\ P_5$ contains $b_i,c_i$ & $v=a,\ P_6$ contains $b_i,c_i$ \\
               \hline
             \end{tabular}
 \end{center}
 \caption{Shortest $d_i-v_j$ paths $(v\in
\{a,b,c\})$}\label{table1}
\end{table}

Moreover, according to the path $P_6$ between $d_i$ and $a_{i+k+1}$,
each pair $(a_l,a_m)$ is also strongly resolved by $d_i$, where
$i+k+1\leq l,m\leq i+n-1\ (l\neq m)$. Finally, each pair
$(a_{i+k},a_j),\ j=i+k+1,i+k+2,\dots,i+n-1$, is strongly resolved by
$a_{i+k-1}$ because of the shortest $a_{i+k-1}-a_j$ path
$a_{i+k-1},a_{i+k},\ldots,a_j$, which contains the vertex $a_{i+k}$.
\end{proof}

\noindent {\it Proof of Theorem \ref{th1}.}\ We have the following
two cases:\\

\noindent Case 1: $n$ is even. Since $lmd(G)\leq \beta(G)$, by
equation $(1)$, and $\beta(S_n) = 3$ \cite{imran} \hphantom{aaaaaaa}
so $lmd(S_n)\leq 3$. For the lower bound, if we suppose that a
subset $\{u,v\}$ \hphantom{aaaaaaa} of $V(S_n)$ is a local resolving
set of $S_n$. Then either both $u$ and $v$ belong \hphantom{aaaaaaa}
to the same cycle $C_i,\ i\in \{a,b,c,d\}$ (but, it is not possible
according to \hphantom{aaaaaaa} Lemma \ref{lem1}), or $u$ and $v$
belong to the different cycles $C_i$ and $C_j$, respectively,
\hphantom{aaaaaaa} where $i,j\in \{a,b,c,d\},\ (i\neq j)$ (but, it
is not possible according to Lemma \hphantom{aaaaaaa} \ref{lem2}).
Hence, no two vertices of $S_n$ form a local resolving set of $S_n$.
Therefore \hphantom{aaaaaaa} $lmd(S_n)\geq 3$.

\noindent Case 2: $n$ is odd. As $lmd(G) = 1$ if and only if $G$ is
a bipartite graph \cite{okamoto} and $S_n$ \hphantom{aaaaaaa} is not
a bipartite graph, so Lemma \ref{lem3} concludes that $lmd(S_n) =
2$.\hphantom{aaaa}$\square$
\\
\\
\noindent {\it Proof of Theorem \ref{th2}.}\ When $n$ is odd, then
Lemma \ref{lem5} and Lemma \ref{lem6} conclude the proof; and when
$n$ is even, then Lemma \ref{lem7} and Lemma \ref{lem8} conclude the
proof.\hphantom{aaaaaaaaaaaaaaaaaaaaaaaaaaaaaaaaaaaaaaaaaaaaaaaaaaaaaaaaaaaaaaaaaa}$\square$

%%%--------------------------------------------------------------------------
\section{Convex Polytopes $U_n$}
The convex polytopes $U_n$, $n\geq 3$, \cite{imran} (Figure
\ref{fig2}) consists of $n$ 4-sided faces, $2n$ 5-sided faces and a
pair of $n$-sided faces obtained by the combination of a convex
polytope $\mathbb{D}_n$ and a prism $D_n$ having vertex and edge
sets as:
$$V(U_n)=\{a_i,b_i,c_i,d_i, e_i\ |\ 1\leq i\leq n\},\hphantom{aaaaaaaaaaaaaaaaaaaaaaaaaaaaaaaaaaaaaaaaaaaaaaa}$$
$$E(U_n)=\{a_ia_{i+1}, b_ib_{i+1}, e_ie_{i+1}, a_ib_i, b_ic_i, c_id_i, d_ie_i, c_{i+1}d_i\ |\ 1\leq i\leq n\}.\hphantom{aaaaaaaaaaaaaaaaaaa}$$

\begin{figure}[h]
        \centerline
        {\includegraphics[width=13cm]{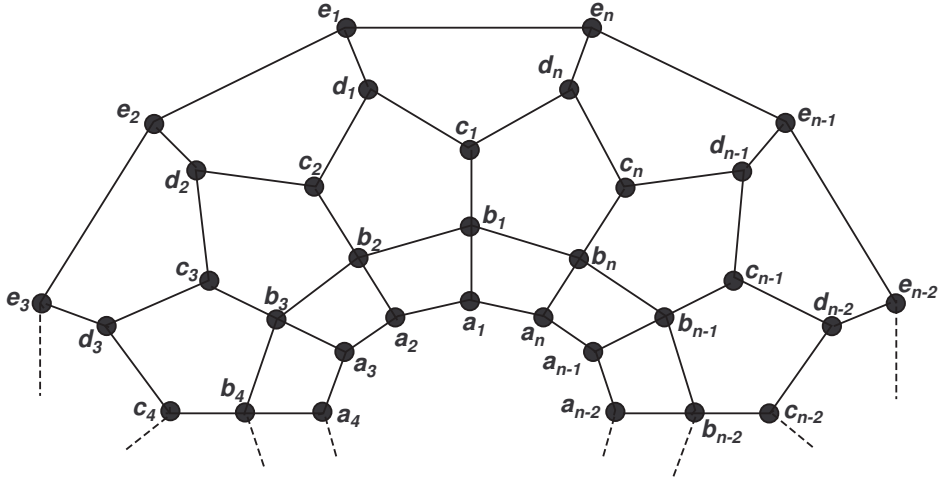}}
        \caption{The graph of convex polytope
        $U_n$}\label{fig2}
\end{figure}
The metric dimension of $U_n$ was studied in \cite{imran}. In this
section, we show that $lmd(U_n) = 2$  for all $n\geq 3$. Moreover,
we show that $sdim(U_n) = 2n$ when $n$ is odd and $sdim(U_n) =
\frac{5n}{2}$ when $n$ is even. The main results of this section are
the following:

\begin{Theorem}\label{th3}
For any convex polytope $U_n$, $n\geq 3$, we have $ lmd(U_n) = 2$.
\end{Theorem}

\begin{Theorem}\label{th4}
For any convex polytope $U_n$, $n\geq 3$, we have
$$ sdim(U_n) =
\left\{
              \begin{array}{ll}
           2n,         &\,\,\,\,\,\,\ \mbox{if}\ n\ \mbox{is odd},\\
           \frac{5n}{2},         &\,\,\,\,\,\,\ \mbox{if}\ n\ \mbox{is
           even}.
            \end{array}
             \right.
$$
\end{Theorem}

Let $V_i= \{i_j\ |\ j=1,2,\ldots,n\},\ i\in \{a,b,c,d,e\}$ be
mutually disjoint subsets of $V(U_n)$. Now, we prove several lemmas
which support the proofs of Theorem \ref{th3} and Theorem \ref{th4}.

\begin{Lemma}\label{lem9}
For $n=2k+1,\ k\geq 1$, $W = \{a_1,a_{k+1}\}$ is a local resolving
set of $U_n$.
\end{Lemma}

\begin{proof}
One can easily see that the list of distances (given below) of each
element of the set $V_i,\ i\in \{a,b,c,d,e\}$, with $a_1$ and
$a_{k+1}$ concludes that for each $u\in N(i_j),\ 1\leq j\leq 2k+1$,
either $d(u,a_1)\neq d(i_j,a_1)$ or $d(u,a_{k+1})\neq
d(i_j,a_{k+1})$, and hance the subset $\{a_1,a_{k+1}\}$ of $V(U_n)$
is a local resolving set of $U_n$, by Proposition \ref{prop1}.
$$ d(a_j,a_l)=
\left\{
              \begin{array}{ll}
           j-1,         &\,\,\,\,\,\,\ 1\leq j\leq k+1;\ l=1,\\
           2k-j+2,       &\,\,\,\,\,\,\ k+2\leq j\leq 2k+1;\ l=1,\\
           k-j+1,         &\,\,\,\,\,\,\ 1\leq j\leq k+1;\ l=k+1,\\
           j-k-1,       &\,\,\,\,\,\,\ k+2\leq j\leq 2k+1;\ l=k+1,\hphantom{aaaaaaaaaaaaaaaaaaaaaaaaaaaaaaaaaaaaaa}
            \end{array}
             \right.
$$
$$ d(b_j,a_l)=
\left\{
              \begin{array}{ll}
           j,         &\,\,\,\,\,\,\ 1\leq j\leq k+1;\ l=1,\\
           2k-j+3,       &\,\,\,\,\,\,\ k+2\leq j\leq 2k+1;\ l=1,\\
           k-j+2,         &\,\,\,\,\,\,\ 1\leq j\leq k+1;\ l=k+1,\\
           j-k,       &\,\,\,\,\,\,\ k+2\leq j\leq 2k+1;\ l=k+1,\hphantom{aaaaaaaaaaaaaaaaaaaaaaaaaaaaaaaaaaaaaa}
            \end{array}
             \right.
$$
$$ d(c_j,a_l)=
\left\{
              \begin{array}{ll}
           j+1,         &\,\,\,\,\,\,\ 1\leq j\leq k+1;\ l=1,\\
           2k-j+4,       &\,\,\,\,\,\,\ k+2\leq j\leq 2k+1;\ l=1,\\
           k-j+3,         &\,\,\,\,\,\,\ 1\leq j\leq k+1;\ l=k+1,\\
           j-k+1,       &\,\,\,\,\,\,\ k+2\leq j\leq 2k+1;\ l=k+1,\hphantom{aaaaaaaaaaaaaaaaaaaaaaaaaaaaaaaaaaaaaa}
            \end{array}
             \right.
$$
$$ d(d_j,a_l)=
\left\{
              \begin{array}{ll}
           j+2,         &\,\,\,\,\,\,\ 1\leq j\leq k+1;\ l=1,\\
           2k-j+4,       &\,\,\,\,\,\,\ k+2\leq j\leq 2k+1;\ l=1,\\
           k-j+3,         &\,\,\,\,\,\,\ 1\leq j\leq k+1;\ l=k+1,\\
           j-k+2,       &\,\,\,\,\,\,\ k+2\leq j\leq 2k+1;\ l=k+1,\hphantom{aaaaaaaaaaaaaaaaaaaaaaaaaaaaaaaaaaaaaa}
            \end{array}
             \right.
$$
$$ d(e_j,a_l)=
\left\{
              \begin{array}{ll}
           j+3,         &\,\,\,\,\,\,\ 1\leq j\leq k+1;\ l=1,\\
           2k-j+5,       &\,\,\,\,\,\,\ k+2\leq j\leq 2k+1;\ l=1,\\
           k-j+4,         &\,\,\,\,\,\,\ 1\leq j\leq k+1;\ l=k+1,\\
           j-k+3,       &\,\,\,\,\,\,\ k+2\leq j\leq 2k+1;\ l=k+1.\hphantom{aaaaaaaaaaaaaaaaaaaaaaaaaaaaaaaaaaaaaa}
            \end{array}
             \right.
$$
\end{proof}

\begin{Lemma}\label{lem10}
For $n=2k,\ k\geq 2$, $W = \{c_1,v\}$ is a local resolving set of
$U_n$, where
$$ v =
\left\{
              \begin{array}{ll}
           e_1,         &\,\,\,\,\,\,\ \mbox{when}\,\ k=2,\\
           c_2,       &\,\,\,\,\,\,\ \mbox{when}\,\ k=3,\\
           d_4,         &\,\,\,\,\,\,\ \mbox{when}\,\ k=4,\\
           c_k,       &\,\,\,\,\,\,\ \mbox{when}\,\ k\geq 5.\hphantom{aaaaaaaaaaaaaaaaaaaaaaaaaaaaaaaaaaaaaa}
            \end{array}
             \right.
$$
\end{Lemma}

\begin{proof}
It is easy to see that the sets $\{c_1,e_1\}, \{c_1,c_2\}$ and
$\{c_1,d_4\}$ are local resolving sets for $U_2,U_3$ and $U_4$,
respectively. For $k\geq 5$, first we give the list of distances of
each vertex of $U_n$ with $c_1$ and $v = c_k$.
$$ d(a_j,c_l)=
\left\{
              \begin{array}{ll}
           j+1,         &\,\,\,\,\,\,\ 1\leq j\leq k;\ l=1,\\
           2k-j+3,       &\,\,\,\,\,\,\ k+1\leq j\leq 2k;\ l=1,\\
           k-j+2,         &\,\,\,\,\,\,\ 1\leq j\leq k;\ l=k,\\
           j-k+2,       &\,\,\,\,\,\,\ k+1\leq j\leq 2k;\ l=k,\hphantom{aaaaaaaaaaaaaaaaaaaaaaaaaaaaaaaaaaaaaa}
            \end{array}
             \right.
$$
$$ d(b_j,c_l) = d(a_j,c_j)-1,\,\ 1\leq j\leq 2k;\ \mbox{for both}\,\ l = 1,k,\hphantom{aaaaaaaaaaaaaaaaaaaaaaaaaaaaaaaaaaaaaa}$$
$$ d(c_j,c_l)=
\left\{
              \begin{array}{ll}
           d(b_j,c_l)-1,         &\,\,\,\,\,\,\ j=1;\ l=1,\\
           d(b_j,c_l),       &\,\,\,\,\,\,\ j=2,2k;\ l=1,\\
           d(b_j,c_l)+1,         &\,\,\,\,\,\,\ j=1,2,2k;\ l=k,\\
           d(b_j,c_l)+1,       &\,\,\,\,\,\,\ 3\leq j\leq k-2;\ l=1,k,\\
           d(b_j,c_l)+1,       &\,\,\,\,\,\,\ j=k-1,k,k+1;\ l=1,\\
           d(b_j,c_l),       &\,\,\,\,\,\,\ j=k-1,k+1,5;\ l=k,\\
           d(b_j,c_l)+1,       &\,\,\,\,\,\,\ k+2\leq j\leq 2k-1;\ l=1,k,\\
           d(b_j,c_l)-1,         &\,\,\,\,\,\,\ j=k;\ l=k.\hphantom{aaaaaaaaaaaaaaaaaaaaaaaaaaaaaaaaaaaaaa}
            \end{array}
             \right.
$$
$$ d(d_j,c_l)=
\left\{
              \begin{array}{ll}
           d(c_j,c_l)+1,         &\,\,\,\,\,\,\ 1\leq j\leq k;\ l=1,\\
           d(c_j,c_l),       &\,\,\,\,\,\,\ k+1\leq j\leq 2k-2;\ l=1,\\
           d(c_j,c_l)-1,       &\,\,\,\,\,\,\ j=2k-1,2k;\ l=1,\\
           d(c_j,c_l)-1,       &\,\,\,\,\,\,\ j=k-2,k-1;\ l=k,\\
           d(c_j,c_l),       &\,\,\,\,\,\,\ 1\leq j\leq k-3\ \wedge\ j=2k;\ l=k,\\
           d(c_j,c_l)+1,       &\,\,\,\,\,\,\ k\leq j\leq 2k-1;\ l=k,\hphantom{aaaaaaaaaaaaaaaaaaaaaaaaaaaaaaaaaaaaaa}
            \end{array}
             \right.
$$
$$ d(e_j,c_l)=
\left\{
              \begin{array}{ll}
           d(d_j,c_l)+1,         &\,\,\,\,\,\,\ j = 1,2k;\ l=1,\\
           d(d_j,c_l),       &\,\,\,\,\,\,\ j = 2,2k-1;\ l=1,\\
           d(d_j,c_l)-1,       &\,\,\,\,\,\,\ 3\leq j\leq 2k-2;\ l=1,\\
           d(d_j,c_l)-1,       &\,\,\,\,\,\,\ 1\leq j\leq k-3\ \wedge\ k+2\leq j\leq 2k;\ l=k,\\
           d(d_j,c_l),       &\,\,\,\,\,\,\ j=k-2,k+1;\ l=k,\\
           d(d_j,c_l)+1,       &\,\,\,\,\,\,\ j= k-1,k;\ l=k,\hphantom{aaaaaaaaaaaaaaaaaaaaaaaaaaaaaaaaaaaaaa}
            \end{array}
             \right.
$$

Now, it can be easily seen that for each $u\in N(i_j),\ 1\leq j\leq
2k+1$, either $d(u,c_1)\neq d(i_j,c_1)$ or $d(u,v)\neq d(i_j,v)$,
and hance, by Proposition \ref{prop1}, the subset $\{c_1,c_k\}$ of
$V(U_n)$ is a local resolving set of $U_n$.
\end{proof}

\begin{Lemma}\label{lem11}
For $n = 2k+1,\ k\geq 1$, if $S$ is a strong resolving set of $U_n$,
then $|S|\geq 2n$.
\end{Lemma}

\begin{proof}
Consider the pair $(a_i,e_{i+k}),\ i=1,2,\ldots,n$, of vertices of
$U_n$. Then it is easy to see that $d(a_i,e_{i+k})=k+4$. As
$diam(U_n)=k+4$, so according to the Property \ref{pro2}, $a_i\in S$
or $e_{i+k}\in S$ for all $i = 1,2,\ldots,n$. Further, the vertices
in the pair $(c_i,d_{i+k}),\ i = 1,2,\ldots,n$ satisfy both the
conditions of Lemma \ref{lem4}, so according to the Property
\ref{pro1}, $c_i\in S$ or $d_{i+k}\in S$ for all $i=1,2,\dots,n$.
Therefore $|S|\geq 2n$.
\end{proof}

\begin{Lemma}\label{lem12}
For $n = 2k+1,\ k\geq 1$, the subset $\{a_i, c_{i}\ |\ i =
1,2,\ldots,n\}$ of $V(U_n)$ is a strong resolving set of $U_n$.
\end{Lemma}

\begin{proof}
First we prove that for each $i=1,2,\ldots,n$, the vertex $a_i$
strongly resolves the pairs $(b_i,b_j), (b_i,d_j)$ and $(b_i,e_j)$
for all $j=1,2\ldots,n$. For this, let us consider the shortest
$a_i-v_j$ paths shown in Table \ref{table2}, where $v\in \{b,d,e\}$.

\begin{table}
 \begin{center}
\begin{tabular}{|c|c|}
  \hline
  for $i\leq j\leq i+k$ & for $i+k\leq j\leq i+n-1$ \\
  \hline
  $a_i,b_i,b_{i+1},\ldots,b_j$ & $b_{j+1},b_{j+2},\ldots,b_{i-1},b_i,a_i$ \\
  $a_i,b_i,b_{i+1},\ldots,b_j,c_j,d_j$ & $d_j,c_{j+1},b_{j+1},b_{j+2},\ldots,b_{i-1},b_i,a_i$ \\
  $a_i,b_i,b_{i+1},\ldots,b_j,c_j,d_j,e_j$ & $e_j,d_j,c_{j+1},b_{j+1},b_{j+2},\ldots,b_{i-1},b_i,a_i$ \\
  \hline
\end{tabular}
 \end{center}
 \caption{Shortest $a_i-v_j$ paths $(v\in
\{b,d,e\})$}\label{table2}
\end{table}

Then each pair $(b_i,b_j), (b_i,d_j)$ and $(b_i,e_j)$ is strongly
resolved by $a_i$ because $b_i$ belongs to each shortest $a_i-v_j$
path listed in Table \ref{table2}, where $v\in \{b,d,e\}$.

Moreover, we note that each pair $(d_i,d_j), (d_i,e_j)$ and
$(e_i,e_j)$ is strongly resolved by $c_i\ (1\leq i\leq n)$ for $j =
i,i+1,\ldots,i+k$, and strongly resolved by $c_{i+1}$ for
$j=i+k+1,\ldots, i+n-1$, because of the following shortest $c_i-v_j$
and $c_{i+1}-v_j$ paths $(v\in \{d,e\})$, which contains the
vertices $d_i$ and $e_i$:\\
\indent $\bullet\ c_i,d_i,c_{i+1},d_{i+1};$\\
\indent $\bullet\ d_{i+n-1},c_i,d_i,c_{i+1};$\\
\indent $\bullet\ c_i,d_i,e_i,e_{i+1},\ldots,e_j,\,\ j =
i,i+1,\ldots, i+k;$\\
\indent $\bullet\ c_i,d_i,e_i,e_{i+1},\ldots,e_j,d_j,\,\ j =
i+2,i+3,\ldots, i+k;$\\
\indent $\bullet\ e_j,e_{j+1},\ldots,e_i,d_i,c_{i+1},\,\ j =
i+k+1,\ldots, i+n-1;$\\
\indent $\bullet\ d_j,e_j,e_{j+1},\ldots,e_i,d_i,c_{i+1},\,\ j =
i+k+1,\ldots, i+n-2.$
\end{proof}

\begin{Lemma}\label{lem13}
For $n = 2k,\ k\geq 2$, if $S$ is a strong resolving set of $U_n$,
then $|S|\geq \frac{5n}{2}$.
\end{Lemma}

\begin{proof}
Consider the pair $(a_i,e_{i+k}),\ i=1,2,\ldots,n$, of vertices of
$U_n$. Then $d(a_i,e_{i+k})=k+3$. Since $diam(U_n)=k+3$ so according
to the Property \ref{pro2}, $a_i\in S$ or $e_{i+k}\in S$ for all $i
= 1,2,\ldots,n$. Moreover, the vertices in the pairs $(c_i,d_{i+2})$
and $(d_j,d_{j+k})$,\ $1\leq i\leq n,\ 1\leq j\leq k$, satisfy both
the conditions of Lemma \ref{lem4}, so according to the Property
\ref{pro1}, $c_i\in S$ or $d_{i+2}\in S$ for all $i=1,2,\dots,n$ and
$d_j\in S$ or $d_{j+k}\in S$ for all $j = 1,2,\ldots,k$. Therefore
$|S|\geq n+n+k = \frac{5n}{2}$.
\end{proof}

\begin{Lemma}\label{lem14}
For $n = 2k,\ k\geq 2$, the subset $\{a_i, c_i, d_{i'}\ |\ i =
1,2,\ldots,n;\ i'=1,2,\ldots,k\}$ of $V(U_n)$ is a strong resolving
set of $U_n$.
\end{Lemma}

\begin{proof}
First we prove that for each $i=1,2,\ldots,n$, the vertex $a_i$
strongly resolves the pairs $(b_i,b_j), (b_i,e_j)$, $1\leq j\leq n$,
and the pair $(b_i,d_j)$ for all $j=k+1,k+2\ldots,n$. For this, let
us consider the shortest $a_i-v_j$ paths shown in Table
\ref{table3}, where $v\in \{b,d,e\}$. Then each pair $(b_i,b_j),
(b_i,d_j)$ and $(b_i,e_j)$ is strongly resolved by $a_i$ because
$b_i$ belongs to each shortest $a_i-v_j$ path listed in Table
\ref{table3}, where $v\in \{b,d,e\}$.

Now, consider the following shortest paths:\\
\indent $\bullet\ c_i,d_i,c_{i+1},d_{i+1},\,\ i = k+1,k+2,\ldots,n;$\\
\indent $\bullet\ d_{i+n-1},c_i,d_i,c_{i+1},\,\ i = k+1,k+2,\ldots,n;$\\
\indent $\bullet\ c_i,d_i,e_i,e_{i+1},\ldots,e_j,\,\ j =
i,i+1,\ldots, i+k-1;$\\
\indent $\bullet\ c_i,d_i,e_i,e_{i+1},\ldots,e_j,d_j,\,\ i =
k+1,k+2,\ldots,n-1;\,\ j =
i+1,i+2,\ldots,i+n-1;$\\
\indent $\bullet\ e_j,e_{j+1},\ldots,e_i,d_i,c_{i+1},\,\ j =
i+k+1,\ldots, i+n-1;$\\
\indent $\bullet\ d_j,e_j,e_{j+1},\ldots,e_i,d_i,c_{i+1},\,\ i =
k+1,k+2,\ldots,n;\,\ j = i+k+1,\ldots, i+n;$\\
\indent $\bullet\ d_i,e_i,e_{i+1},\ldots,e_{i+k},\,\
i=1,2,\ldots,k;$\\
\indent $\bullet\ c_i,d_i,e_i,e_{i+1},\ldots,e_j,d_j,\,\ i =
k+1,k+2,\ldots,n;\,\ j =
i,i+1,\ldots,i+k-1;$\\
\indent $\bullet\ d_{i+k},e_{i+k},e_{i+k+1},\ldots,
e_{i-1},e_i,d_i,\,\ i=k+1,k+2,\ldots,n$.\\
Then it is easy to see that the pairs $(d_i,d_j);\ k+1\leq i\leq
n-1, i+1\leq j\leq n$, $(d_i,e_j);\ k+1\leq i\leq n, i\leq j\leq
i+k-1$, $(e_i,e_j);\ 1\leq i\leq n, i+1,\leq j\leq i+k-1$ are
strongly resolved by $c_i$; the pairs $(d_i,e_j);\ k+1\leq i\leq
n-1, i+k+1\leq j\leq i+n-1$, $(e_i,e_j);\ 1\leq i\leq n, i+k+1\leq
j\leq i+n-1$ are strongly resolved by $c_{i+1}$; the pair
$(d_i,e_{i+k});\ k+1\leq i\leq n$ is strongly resolved by $d_{i+k}$;
and the pair $(e_i,e_{i+k});\ 1\leq i\leq k$ is strongly resolved by
$d_i$, because the above listed shortest $c_i-v_j,\ c_{i+1}-v_j;\
v\in \{d,e\}$, $d_i-e_{i+k}$ and $d_{i+k}-d_i$ paths contains the
vertices $d_i,e_i$ and $e_{i+k}$.

\begin{table}
 \begin{center}
\begin{tabular}{|c|c|}
  \hline
  for $i\leq j\leq i+k-1$ & for $i+k\leq j\leq i+n-1$ \\
  \hline
  $a_i,b_i,b_{i+1},\ldots,b_j$ & $b_{j+1},b_{j+2},\ldots,b_{i-1},b_i,a_i$ \\
  $a_i,b_i,b_{i+1},\ldots,b_j,c_j,d_j$ & $d_j,c_{j+1},b_{j+1},b_{j+2},\ldots,b_{i-1},b_i,a_i$ \\
  $a_i,b_i,b_{i+1},\ldots,b_j,c_j,d_j,e_j$ & $e_j,d_j,c_{j+1},b_{j+1},b_{j+2},\ldots,b_{i-1},b_i,a_i$ \\
  \hline
\end{tabular}
 \end{center}
 \caption{Shortest $a_i-v_j$ paths $(v\in
\{b,d,e\})$}\label{table3}
\end{table}
\end{proof}

\noindent {\it Proof of Theorem \ref{th3}.}\ As $lmd(G) = 1$ if and
only if $G$ is a bipartite graph \cite{okamoto} and $U_n$ is not a
bipartite graph, so Lemma \ref{lem9} and Lemma \ref{lem10} conclude
that $lmd(U_n) =
2$.\hphantom{aaaaaaaaaaaaaaaaaaaaaaaaaaaaaaaaaaaaaaaaaaaaaaaaaaaaaaaaaaaaaaaaaaaaaa}$\square$
\\
\\
\noindent {\it Proof of Theorem \ref{th4}.}\ When $n$ is odd, then
Lemma \ref{lem11} and Lemma \ref{lem12} conclude the proof; and when
$n$ is even, then Lemma \ref{lem13} and Lemma \ref{lem14} conclude
the
proof.\hphantom{aaaaaaaaaaaaaaaaaaaaaaaaaaaaaaaaaaaaaaaaaaaaaaaaaaaaaaaaaaaaaaaaaa}$\square$
%----------------------------------------------------------------------------


\begin{thebibliography}{999}
\bibitem{bailey1}
R. Bailey, P. Cameron, Basie size, metric dimension and other
invariants of groups and graphs, Bull. of London Math. Soc. 43(2011)
209-242.
\bibitem{bailey2}
R. Bailey, K. Meagher, On the metric dimension of grassmann graphs,
Technical Reports 2011, arXiv: 1010.4495.
\bibitem{beer}
Z. Beerloiva, F. Eberhard, T. Erlebach. A. Hall, M. Hoffmann, M.
Mihal\'{a}k, L. Ram, Network discovery and verification. IEEEE J.
Selected Area in Commun. 24(2006) 2168-2181.
\bibitem{caceres1}
J. C\'{a}ceres, C. Hernando, M. Mora, I. M. Pelayoe, M. L. Puertas,
On the metric dimension of infinite graphs, Elect. Notes in Disc.
Math. 35(2009) 15-20.
\bibitem{caceres2}
J. C\'{a}ceres, C. Hernando, M. Mora, I. M. Pelayoe, M. L. Puertas,
C. Seara, D. R. Wood, On the metric dimension of cartesian products
of graphs, SIAM J. Disc. Math. 21(2007) 423-441.
\bibitem{chappell}
G. Chappell, J. Gimbel, C. Hartman, Bounds on the metric and
partition dimension of a graph, Ars Combin. 88(2008) 349-366.
\bibitem{chartrand6}
G. Chartrand, P. Zhang, The theory and applications of resolvability
in graphs: A survey, Congr. Numer. 160(2003) 47-68.
\bibitem{chartrand}
G. Chartrand, L. Eroh, M. A. Johnson, O. R. Oellermann,
Resolvability in graphs and the metric dimension of a graph, Disc.
Appl. Math. 105(2000) 99-113.
\bibitem{currie}
J. Currie, O. Oellermann, The metric dimension and metric
independence of a graph, J. Combin. Math. Combin. Comput. 39(2001)
157-167.
\bibitem{cvetk}
D. Cvetkovi\'{c}, P. Hansen, V. Kova\v{c}evi\'{c}-Vuj\v{c}i\'{c}, On
some interconnections between combinatorial optimization and
extremal hraph theory, Yugosal. J. Oper. Res. 14(2)(2004) 147-154.
\bibitem{fehr}
M. Fehr, S. Gosselin, O. Oellermann, The metric dimension of Cayley
digraphs, Disc. Math. 306(2006) 31-41.
\bibitem{harary}
F. Harary, R. A. Melter, On the metric dimension of a graph, Ars.
Combin. 2(1976) 191-195.
\bibitem{hernando1}
C. Hernando, M. Mora, P. J. Slater, D. R. Wood, Fault-Tolerant
metric dimension of graphs, Proc. Internat. Conf. Convexity in
Discrete Structures, Ramanujan Math. Society Lecture Notes, 5(2008)
81-85.
\bibitem{imran}
M. Imran. S. U. H. Bokary, A. Baig, On families of convex polytopes
with constant metric dimension, Comp. Math. Appl. 60(2010)
2629-2638.
\bibitem{javaid4}
I. Javaid, M. Salman, M. A. Chaudhry, S. Shokat, Fault-Tolerance in
Resolvability, Util. Math. 80(2009) 263-275.
\bibitem{javaid1}
I. Javaid, M. T. Rahim, K. Ali, Families of regular graphs with
constant metric dimension, Util. Math. 75(2008) 21-33.
\bibitem{javaid2}
I. Javaid, M. Salman, M. A. Chaudary, S. A. Aleem, On the metric
dimension of generlized Petersen graphs, Quaestiones Mathematicae,
in press.
\bibitem{javaid3}
I. Javaid, M. N. Azhar M. Salman, Metric dimension and determining
number of Cayley graphs, World Applied Sciences Journal, in press.
\bibitem{khuller}
S. Khuller, B. Raghavachari, A. Rosenfeld, Landmarks in graphs,
Disc. Appl. Math. 70(1996) 217-229.
\bibitem{kratica}
J. Kratica, V. Kova\v{c}evi\'{c}-Vuj\v{c}i\'{c}, M.
\v{C}angalovi\'{c}, M. Stojanovi\'{c}. Minimal doubly resolving sets
and the strong metric dimension of some convex polytopes, Appl.
Math. Comp. 218(2012) 9790-9801.
\bibitem{kratica1}
J. Kratica, M. \v{C}angalovi\'{c}, V.
Kova\v{c}evi\'{c}-Vuj\v{c}i\'{c}, Computing minimal doubly resolving
sets of graphs, Comp. Oper. Res. 36(2009) 2149-2159.
\bibitem{liu}
K. Liu, N. Abu-Ghazaleh, Virtual coordinate back tracking for void
travarsal in geographic routing, Lecture Notes Comp. Sci. 4104(2006)
46-59.
\bibitem{manuel}
P. Manuel, I. Rajasingh, Minimum metric dimension of silicate
networks, Ars. Combin. 98(2011) 501-510.
\bibitem{may}
T. May, O. Oellermann, The strong dimension of distance hereditary
graphs, J. Combin. Math. Combin. Comput. 76(2011) 59-73.
\bibitem{mlad}
N. Mladenovi\'{c}, J. Kratica, V. Kova\v{c}evi\'{c}-Vuj\v{c}i\'{c},
M. \v{C}angalovi\'{c}, Variable neighborhood search for metric
dimension and minimal doubly resolving set problems, Europ. J. Oper.
Res. 220(2012) 328-337.\bibitem{okamoto} F. Okamoto, L. Crosse, B.
Phinezy, P. Zhang, Kalamazoo, The local metric dimension of graphs,
Mathematica Bohemica 135(3)(2010) 239-255.

\bibitem{oellermann}
O. Oellermann, J. Peters-Fransen, The strong metric dimension of
graphs and digraphs, Disc. Appl. Math. 155(2007) 356-364.
\bibitem{salman}
M. Salman, I. Javaid, M. A. Chaudhry, Resolvability in circulant
graphs, Acta. Math. Sinica Englich Series 28(9)(2012) 1851-1864.
\bibitem{sebo}
A. Seb\"{o}, E. Tannier, On metric generators of graphs, Math. Oper.
Res. 29(2)(2004) 383-393.
\bibitem{slater}
P. J. Slater, Leaves of trees, Cong. Numer. 14(1975) 549-559.
\end{thebibliography}
\end{document}